\documentclass[a4paper]{amsart}

\title[Polish groups with metrizable UMF]{Polish groups with metrizable universal minimal flows}

\author[J.~Melleray]{Julien Melleray}
\address{Universit\'e de Lyon \\
  CNRS UMR 5208, Universit\'e Lyon 1 \\
  Institut Camille Jordan \\
  43 blvd. du 11 novembre 1918 \\
  F-69622 Villeurbanne \sc{cedex}, France}
\email{melleray@math.univ-lyon1.fr}

\author[L.~Nguyen~Van~Th\'e]{Lionel Nguyen~Van~Th\'e}
\address{Aix-Marseille Universit\'e, CNRS \\
  Centrale Marseille, I2M, UMR 7373 \\ 
  13453 Marseille, France} 
\email{lionel.nguyen-van-the@univ-amu.fr}

\author[T.~Tsankov]{Todor Tsankov}
\address{Institut de Math\'ematiques de Jussieu--PRG \\
Universit\'e Paris Diderot, Case 7012 \\
75205 Paris \sc{cedex} 13, France}
\email{todor@math.univ-paris-diderot.fr}

\thanks{Research partially supported by the ANR project GrupoLoco (ANR-11-JS01-0008).}

\usepackage[initials,shortalphabetic]{amsrefs}

\usepackage{ifxetex}
\ifxetex
\usepackage[T1]{fontenc}
\usepackage{fourier}
\usepackage{xunicode}
\usepackage[no-math]{fontspec}
\else
\usepackage[utf8]{inputenc}
\usepackage[T1]{fontenc}
\usepackage{fourier}
\fi

\usepackage{amssymb}
\usepackage{eucal} 

\usepackage[all]{xy}

\usepackage{enumitem}
\setlist[enumerate,1]{label=(\roman*), font=\normalfont}

\usepackage{my-macros}

\numberwithin{equation}{section}

\subjclass[2010]{Primary: 37B05 
; Secondary: 03C15 
03E02 
03E15 
05C55 
22F50 
43A07 
54H20 
}
\keywords{Non locally compact transformation groups, universal minimal flow, universal minimal proximal flow, extreme amenability, strong amenability, Ramsey theory}

\begin{document}
\maketitle

\begin{abstract}
  We prove that if the universal minimal flow of a Polish group $G$ is metrizable and contains a $G_\delta$ orbit $G \cdot x_0$, then it is isomorphic to the completion of the homogeneous space $G/G_{x_0}$ and show how this result translates naturally in terms of structural Ramsey theory. We also investigate universal minimal proximal flows and describe concrete representations of them in a number of examples.
\end{abstract}

\section{Introduction}
The connections between structural Ramsey theory and topological dynamics first became apparent in the work of Pestov~\cite{Pestov1998}, where, using the classical Ramsey theorem, he showed that the automorphism group of the dense countable linear order $(\Q, <)$ is \df{extremely amenable} (i.e., every time it acts continuously on a compact space, there is a fixed point) and produced the first interesting example of a non-trivial, metrizable, universal minimal flow. Later, Glasner and Weiss~\cite{Glasner2002a}, again using the Ramsey theorem, proved that the space of all linear orderings on a countable set is the universal minimal flow of the infinite permutation group; then, they also calculated the universal minimal flow of the homeomorphism group of the Cantor space \cite{Glasner2003a}. Inspired by those results, Kechris, Pestov, and Todorcevic~\cite{Kechris2005}, using the general framework of \emph{Fraïssé limits}, formulated a precise correspondence between the structural Ramsey property for a Fraïssé class and the extreme amenability of the corresponding automorphism group. They also developed a general method for calculating universal minimal flows and their work spawned a renewed interest in structural Ramsey theory. 

It turns out that for most Fraïssé classes that have been considered in the literature, even if they do not have the Ramsey property, they are often not far from having it: namely, it is possible to expand the class by an ordering, and perhaps some additional structure, in a way that the resulting Fraïssé class does have the Ramsey property. A corollary of the main result of this paper, Theorem~\ref{t:met1}, is a characterization, in terms of topological dynamics, of when this happens.

An important motivation for our work was the following question, raised in \cite{Bodirsky2013}, whether, under rather general conditions, a Ramsey expansion always exists.

\begin{question} \label{q:olig-Ramsey}
Let $\bF$ be a Fraïssé structure whose age has only finitely many non-isomorphic structures in every finite cardinality (equivalently, such that the action $\Aut(\bF) \actson \bF$ is oligomorphic). Does $\bF$ always admit a Ramsey precompact expansion?  
\end{question}

Even though initially we were mostly motivated by Ramsey theory for discrete structures (that is, automorphism groups that are subgroups of $S_\infty$), our methods and results are more naturally placed in the more general setting of Polish groups.

If $G$ is a topological group, a \df{$G$-flow} is a compact Hausdorff space $X$ equipped with a continuous action of $G$. A flow is \df{minimal} if it has no proper subflows; it is a standard fact in topological dynamics, due to Ellis, that every topological group $G$ has a universal minimal flow $M(G)$ such that every other minimal $G$-flow is a factor of it, and that it is unique up to isomorphism.

If $G$ is a topological group and $H$ is a closed subgroup, the homogeneous space $G/H$ is equipped with a natural uniformity (the quotient of the right uniformity of $G$), defined as follows: entourages of the diagonal $U_V$ are indexed by symmetric neighborhoods $V$ of $1_G$ and are given by
\begin{equation} \label{eq:unif}
U_V = \set{(gH, vgH) : g \in G, v \in V}.
\end{equation}
Note that this uniformity is compatible with the quotient topology on $G/H$ and, in the case where $G$ is Polish, it is metrizable by the distance
\begin{equation} \label{eq:dist}
d(g_1H, g_2H) = \inf_{h \in H} d_R(g_1h, g_2)
\end{equation}
(where $d_R$ is some right-invariant distance on $G$), but usually not complete. We will denote by $\widehat{G/H}$ the completion of $G/H$ and will say that $G/H$ is \df{precompact} (or sometimes, that $H$ is \df{co-precompact} in $G$) if $\widehat{G/H}$ is compact. Equivalently, $H$ is co-precompact in $G$ iff for every neighborhood $V \ni 1_G$, there exists a finite set $F \sub G$ such that $VFH = G$.

In the situation where $G$ is the automorphism group of a Fraïssé structure $\bF$ and $G^*$ is the automorphism group of some homogeneous expansion $\bF^*$ of $\bF$, it is not difficult to see that $\widehat{G/G^*}$ can be identified with the space $X$ of all expansions of $\bF$ whose age is contained in the age of $\bF^*$ with the topology given by the basis of sets of the form
\[
U_{\bA, \bA^*} = \set{x \in X : x|_\bA = \bA^*},
\]
where $\bA$ is a finite substructure of $\bF$ and $\bA^*$ is an expansion of $\bA$ in the age of $\bF^*$. Thus we see that $G/G^*$ is precompact iff every finite substructure $\bA \sub \bF$ has only finitely many expansions in the age of $\bF^*$. For example, if $G \actson \bM$ is an oligomorphic permutation group, $G^*$ is co-precompact in $G$ iff the action $G^* \actson \bM$ is oligomorphic. See \cite{NguyenVanThe2013} for more on precompact expansions.

Observe that if $G/H$ is precompact, then $\widehat{G/H}$ is a $G$-flow and, as $G/H$ is a Polish space and embeds homeomorphically in its completion, it is always a dense $G_\delta$ subset of $\widehat{G/H}$. The flow $\widehat{G/H}$ has the following universal property: if $G \actson X$ is any $G$-flow and $x_0 \in X$ is a point fixed by $H$, then there is a unique morphism of $G$-flows $\pi \colon \widehat{G/H} \to X$ such that $\pi(H) = x_0$. (This is true simply because the map $G/H \to X$, $gH \mapsto g \cdot x_0$ is uniformly continuous.) If it happens that the group $H$ is extremely amenable, then this implies that the flow $\widehat{G/H}$ is universal in the sense that it maps to every other $G$-flow. If, in addition, the flow $\widehat{G/H}$ is minimal, then it is the universal minimal flow of $G$. This technique for calculating universal minimal flows was first considered by Pestov in \cite{Pestov2006}*{Section~4.5}. It is worth pointing out that \emph{all} known metrizable universal minimal flows of Polish groups are of this form and it is an open question whether it is true in general. In the following theorem, we provide a positive answer, under the additional assumption of the existence of a $G_\delta$ orbit.
\begin{theorem}
\label{t:met1}
Let $G$ be a Polish group and $M(G)$ be its universal minimal flow. Then the following are equivalent: 
\begin{enumerate}
\item \label{met1:i:1} The flow $M(G)$ is metrizable and has a $G_\delta$ orbit. 
\item \label{met1:i:2} There is a closed, co-precompact, extremely amenable subgroup $G^{*}\leq G$ such that $M(G)=\widehat{G/G^{*}}$. 
\end{enumerate}
\end{theorem}

In the setting of Fraïssé limits, Theorem~\ref{t:met1} translates to the following.
\begin{cor}
\label{c:met2}
Let $\bF$ be a Fraïssé structure and let $G=\Aut(\bF)$. The following are equivalent: 
\begin{enumerate}
\item The flow $M(G)$ is metrizable and has a $G_\delta$ orbit.
\item The structure $\bF$ admits a Fraïssé precompact expansion $\bF^{*}$ whose age consists of rigid elements, has the Ramsey property, and has the expansion property relative to $\Age(\bF)$.
\end{enumerate}
\end{cor}
See \cite{NguyenVanThe2013} for details and for the definition of the Ramsey and the expansion property. 

It is natural to ask whether the assumption of the existence of the $G_{\delta}$ orbit can be omitted from item \ref{met1:i:1} of Theorem~\ref{t:met1}. In fact, this question had already been raised by Angel, Kechris, and Lyons in \cite{Angel2012}:
\begin{question}
\label{q:metrizable-generic}
  If the universal minimal flow $M(G)$ of a Polish group $G$ is metrizable, does $M(G)$ necessarily have a $G_\delta$ orbit?
\end{question}
  While this paper was being completed, Andy Zucker informed us that he had proved \cite{Zucker2014p}, in the important case where $G$ is a subgroup of $S_\infty$, a result stronger than Theorem~\ref{t:met1} in which the hypothesis of the existence of a $G_\delta$ orbit can be omitted from \ref{met1:i:1}, thus answering Question~\ref{q:metrizable-generic} in this case. His work was independent from ours and his methods are quite different. Question~\ref{q:metrizable-generic} remains open in general.

  Zucker's result also implies that Question~\ref{q:olig-Ramsey} is equivalent to the question whether every oligomorphic permutation group has a metrizable universal minimal flow. It was shown in \cite{Tsankov2012} that oligomorphic permutation groups are (more or less) the Roelcke precompact subgroups of $S_\infty$. (A topological group $G$ is \df{Roelcke precompact} if for every neighborhood $U$ of $1_G$, there is a finite set $F \sub G$ such that $UFU = G$.) It seems plausible that Roelcke precompactness alone may imply the metrizability of the universal minimal flow. 
\begin{question}
\label{q:Roelke-metrizable}
Is $M(G)$ metrizable for every Roelcke precompact Polish group $G$?
\end{question}

A positive answer to Questions \ref{q:olig-Ramsey} and/or \ref{q:Roelke-metrizable} would show that Ramsey classes are rather ubiquitous objects and not at all exceptional, as was initially believed. Since \cite{Kechris2005}, a number of new precompact Ramsey expansions have been found for various kinds of Fraïssé classes: metric spaces with rational distances \cite{Nesetril2005}, all classes of posets (and essentially all classes of undirected graphs) \cites{Sokic2012, Sokic2012a}, $\omega$-categorical linear orders \cite{Dorais2013}, boron trees \cite{Jasinski2013}, and all classes of directed graphs \cite{Jasinski2014}. Those results could be considered as evidence for a positive answer to Question~\ref{q:olig-Ramsey}. On the other hand, several natural problems about the Ramsey property remain open: finite metric spaces with distances in some fixed set, Euclidean metric spaces (this problem is mentioned in \cite{Kechris2005}), projective Fraïssé classes (those are developed in \cite{Irwin2006} and are connected to Fraïssé classes of finite Boolean algebras) and equidistributed Boolean algebras (this problem appears in \cite{Kechris2012p}).

Flows of the type $\widehat{G/H}$ are usually not difficult to understand and when it happens that $M(G) = \widehat{G/G^*}$ for some closed, co-precompact $G^* \leq G$, this provides ample information for the dynamical properties of all minimal $G$-flows. Recall that a $G$-flow $X$ is called \df{coalescent} if every endomorphism of $X$ is an automorphism. The universal minimal flow $M(G)$ is always coalescent (this fact is due to Ellis; for a proof, see for example \cite{Uspenskij2000}*{Proposition 3.3}) but if $M(G) = \widehat{G/G^*}$, much more is true.
\begin{theorem}
  \label{th:intro-coalescent}
  Let $G$ be a Polish group. Then the following statements hold:
  \begin{enumerate}
  \item \label{i:6} Every minimal $G$-flow of the form $\widehat{G/H}$ is coalescent and has a compact automorphism group;
  \item If $M(G) = \widehat{G/G^*}$ for some closed, co-precompact $G^* \leq G$, then the conclusion of \ref{i:6} is true for every minimal $G$-flow.
  \end{enumerate}
\end{theorem}

Another consequence is the fact that if $M(G) = \widehat{G/G^*}$, then minimal flows of $G$ are easy to classify, at least in the sense of descriptive set theory: in Subsection~\ref{sec:smooth-isom}, we show that the equivalence relation of isomorphism of minimal flows of $G$ is \emph{smooth}.

Next we show how to calculate the \emph{universal minimal proximal flow} of groups for which $M(G)$ is of the form $\widehat{G/G^*}$ and isolate a criterion for such a group $G$ to be \emph{strongly amenable} (see Section~\ref{sec:proximal-flows} for the definitions).
\begin{theorem}
  \label{th:univ-prox}
  Let $G$ be a Polish group and $G^*$ a closed, co-precompact subgroup such that $M(G) = \widehat{G/G^*}$. Let $N(G^*)$ denote the normalizer of $G^*$ in $G$. Then the following statements hold:
  \begin{enumerate}
  \item The universal minimal proximal flow of $G$ is isomorphic to $\widehat{G/N(G^*)}$;
  \item $G$ is strongly amenable iff $G^*$ is normal in $G$.
  \end{enumerate}
\end{theorem}
In the case of automorphism groups of Fraïssé structures, Theorem~\ref{th:univ-prox} can be used to characterize strong amenability in Ramsey-theoretic terms. 

The paper is organized as follows. In Section~\ref{sec:minimal-flows-form}, we study minimal flows with a $G_\delta$ orbit and prove Theorem~\ref{th:intro-coalescent}; in Section~\ref{sec:metrizable-flows}, we prove Theorem~\ref{t:met1}; in Section~\ref{sec:proximal-flows}, we discuss proximal flows and prove Theorem~\ref{th:univ-prox}; and, finally, in Section~\ref{sec:examples}, we calculate the universal minimal proximal flow for several examples.

\subsection*{Acknowledgements}
L.N.V.T. would like to acknowledge the support of the CNRS and the hospitality of the Institut Camille Jordan (Universit\'e Lyon 1), and to thank Bohuslav Balcar for pointing out the notion of strong amenability, as well as Gregory Cherlin, Arnaud Hilion, and Mauro Mariani for helpful discussions. T.T. would like to thank the Caltech mathematics department, during a visit to which part of this work was done as well as Alexander Kechris for useful discussions. The last stage of the project was completed during the program on Universality and Homogeneity held at the Hausdorff Research Institute for Mathematics in Bonn, and we would like to thank the Institute and the organizers of the program for having made this possible. We are also grateful to the anonymous referees for catching some imprecisions as well as making suggestions that helped us improve the exposition.


\section{Coalescence and automorphisms of minimal flows}
\label{sec:minimal-flows-form}
In this section, we study endomorphisms of minimal flows that have a $G_\delta$ orbit with a co-precompact stabilizer. Throughout $G$ will be a Polish group.

We start with the following well-known fact.
\begin{lemma}
  \label{l:Gdelta-orbit}
  Let $G \actson X$ be a metrizable minimal flow with a $G_\delta$ orbit $G \cdot x_0$, $Y$ be a minimal $G$-flow, and let $\pi \colon X \to Y$ be a factor map. Then the orbit $G \cdot \pi(x_0)$ is also $G_\delta$. In particular, if $\pi \colon X \to X$ is an endomorphism, then $\pi(G \cdot x_0) = G \cdot x_0$.
\end{lemma}
\begin{proof}
  The first assertion follows, for example, from \cite{Melleray2013}*{Proposition~A.7}. The second is a consequence of the Baire category theorem.
\end{proof}

Recall that if $H \leq G$ is a closed subgroup, the homogeneous space $G/H$ is equipped with a natural distance defined by \eqref{eq:dist}. The next lemma shows that under a precompactness assumption, all endomorphisms of the homogeneous $G$-space $G/H$ are in fact isometries for this distance.
\begin{lemma}
  \label{l:homog-Gmap}
  Let $G$ be a Polish group and $H \leq G$ a closed, co-precompact subgroup. Suppose $\phi \colon G/H \to G/H$ is a $G$-map. Then there exists $f_0 \in N(H)$ such that $\phi(gH) = gf_0H$ for all $g \in G$. In particular, $\phi$ is an isometry.
\end{lemma}
\begin{proof}
  Let $f_0 \in G$ be such that $\phi(H) = f_0H$. The set of all $g\in G$ fixing $\phi(H)$ is $f_{0}Hf_{0} ^{-1}$. Furthermore, as $\phi$ is a $G$-map, $\phi(H)$ is fixed by $H$, i.e., $H \sub f_0 H f_0^{-1}$, or, which is the same, $f_{0} ^{-1}H f_{0}\sub H$. For the reverse inclusion, we first check that $\phi$ is a contraction for the distance $d$ defined by \eqref{eq:dist}: for any $g_1, g_2 \in G$ and $h \in H$, we have
\begin{equation*} 
  \begin{split}
    d \big(\phi(g_1H), \phi(g_2H) \big) &= d(g_1 f_0 H, g_2 f_0 H) \\
    &= d(g_1 f_0 (f_0^{-1} h f_0) H, g_2 f_0 H) \\
    &\leq d_R(g_1 h f_0, g_2 f_0) \\
    &= d_R(g_1h, g_2),
  \end{split} 
\end{equation*}
  so
  \[
  d \big( \phi(g_1H), \phi(g_2H) \big) \leq \inf_{h \in H} d_R(g_1h, g_2) = d(g_1H, g_2H).  
  \]
  The map $\phi$ extends to a surjective contraction of the compact metric space $\widehat{G/H}$ and is therefore an isometry (see \cite{Engelking1989}*{Exercise 4.5.4, p.~289}). In particular, $\phi$ is injective, so for any $g\in G$, if $g$ fixes $\phi(H)$, then $gH=H$, and $g\in H$. Hence, $f_0 H f_0^{-1}\sub H$, completing the proof.
\end{proof}

Observe that if $H \leq G$ is co-precompact, then $N(H)/H$ is a Polish group whose right uniformity is precompact (as a subspace of the precompact space $G/H$) and therefore compact (see, for instance, \cite{Solecki2002}*{Lemma 1.2}).

If $X$ is a $G$-flow, the \df{automorphism group} of $X$ is the topological group
\[
\Aut(X) = \set{\gamma \in \Homeo(X) : \gamma(g \cdot x) = g \cdot \gamma(x)
  \text{ for all } g \in G, x \in X}.
\]
$\Aut(X)$ is a closed subgroup of $\Homeo(X)$ (the latter being equipped with the uniform convergence topology) and the action $\Aut(X) \actson X$ is continuous. 

\begin{theorem}
  \label{th:coalescent}
Let $G$ be a Polish group, $G \actson X$ be a minimal $G$-flow with a $G_\delta$ orbit $G \cdot x_0$ such that the stabilizer $H = G_{x_0}$ is co-precompact. Then $X$ is coalescent and $\Aut(X)$ embeds naturally as a closed subgroup of the compact group $N(H)/H$. If $X \cong \widehat{G/H}$, then this embedding is an isomorphism.
\end{theorem}
\begin{proof}
  First consider the \df{ambit} (i.e., flow with a distinguished point with a dense orbit) $(Y, y_0) = (\widehat{G/H}, H)$. For $f \in N(H)$, define $\theta_f \colon G \cdot y_0 \to G \cdot y_0$ by $\theta_f(g \cdot y_0) = gf^{-1} \cdot y_0$. This is an isometry of $G/H$ (with the distance \eqref{eq:dist}) that commutes with the action of $G$ and therefore extends to an automorphism $Y \to Y$, still denoted by $\theta_f$. The map $\Phi \colon N(H) \to \Aut(Y)$, $\Phi(f) = \theta_f$ is a homomorphism whose kernel is $H$. By Lemma~\ref{l:homog-Gmap}, $\Aut(Y)$ is a subgroup of the isometry group $\Iso(Y)$, where the pointwise convergence and uniform convergence topologies coincide (see \cite{Pestov2006}*{Proposition 5.2.1}). The map $\Phi$, being obviously continuous for the former, is therefore also continuous for the latter. Lemma~\ref{l:homog-Gmap} also implies that $\Phi$ is surjective. We conclude that $\Aut(Y) \cong N(H)/H$.

  Now let $\phi \colon X \to X$ be an endomorphism. Denote by $\pi \colon Y \to X$ the factor map given by $\pi(y_0) = x_0$ and note that the existence of this map implies that $X$ is metrizable. By Lemma~\ref{l:Gdelta-orbit}, $\phi(G \cdot x_0) = G \cdot x_0$. By Lemma~\ref{l:homog-Gmap}, identifying $G\cdot x_{0}$ and $G/H$, there exists $f_\phi \in N(H)$ such that $\phi(x_0) = f_\phi^{-1} \cdot x_0$. Denote by $\theta$ the automorphism $\theta_{f_\phi} \in \Aut(Y)$ and consider the diagram
\[ \xymatrix{
Y \ar[d]_\pi \ar[r]^{\theta_{f_{\phi}}} &Y \ar[d]^\pi \\
X \ar[r]^\phi & X.
} \]
It commutes on the dense set $G \cdot y_0$ and therefore everywhere.

We proceed to show that $\phi$ is injective, that is $\phi \in \Aut(X)$. Let $\mcR_\pi$ be the closed equivalence relation on $Y$ defined by
\[
y_1 \eqrel{\mcR_\pi} y_2 \iff \pi(y_1) = \pi(y_2).
\]
As $\phi$ is an endomorphism of $X$ and the diagram commutes, we have
\begin{equation} \label{eq:theta-end}
\forall y_1, y_2 \in Y \quad y_1 \eqrel{\mcR_\pi} y_2 \implies \theta(y_1) \eqrel{\mcR_\pi} \theta(y_2).
\end{equation}
Now we show the converse. Let $y_1, y_2 \in Y$ be such that $\theta(y_1) \eqrel{\mcR_\pi} \theta(y_2)$. As the group $\Aut(Y)$ is compact, there exists a sequence of positive integers $\set{n_k}_k$ such that $\theta^{n_k} \to \id$. Applying $\theta^{n_k-1}$ to both sides of the expression $\theta(y_1) \eqrel{\mcR_\pi} \theta(y_2)$, using \eqref{eq:theta-end}, and taking limits, we obtain that $y_1 \mcR_\pi y_2$, thus showing that $\phi$ is injective.

It is now not difficult to check that the map $F \colon \Aut(X) \to \Aut(Y)$ given by $F(\phi) = \theta_{f_\phi}$ is a well-defined, injective group homomorphism. That $F$ is a topological embedding follows from the identity $\phi \circ \pi = \pi \circ F(\phi)$ and the fact that $\pi$ is surjective and uniformly continuous. This completes the proof of the theorem.
\end{proof}

\begin{cor}
  \label{c:all-coalescent}
  Let $G$ be a Polish group such that $M(G) = \widehat{G/G^*}$ for some closed, co-precompact $G^* \leq G$. Then every minimal $G$-flow is coalescent and has a compact automorphism group.
\end{cor}
\begin{proof}
  By the universality of $M(G)$ and Lemma~\ref{l:Gdelta-orbit}, every minimal flow $G \actson X$ has a $G_\delta$ orbit $G \cdot x_0$. Then $G_{x_0} \geq G^*$ is co-precompact in $G$ and Theorem~\ref{th:coalescent} yields the conclusion.
\end{proof}

The next lemma, which will be useful in the next section, describes completely the $H$-fixed points in minimal flows of the form $\widehat{G/H}$.
\begin{lemma} \label{l:fixed-pts}
  Let $G$ be a Polish group, $H \leq G$ a closed, co-precompact subgroup such that the $G$-flow $\widehat{G/H}$ is minimal. Denote by $x_0$ the point $H \in \widehat{G/H}$. Then
  \[
  \set{x \in \widehat{G/H} : H \cdot x = x} = N(H) \cdot x_0.
  \]
\end{lemma}
\begin{proof}
  Let $x \in \widehat{G/H}$ be an $H$-fixed point. By the universal property of $\widehat{G/H}$, there exists a $G$-map $\pi \colon \widehat{G/H} \to \widehat{G/H}$ such that $\pi(x_0) = x$. By Lemma~\ref{l:homog-Gmap}, there is $g \in N(H)$ such that $\pi(x_0) = g \cdot x_0$. The other inclusion is easy. 
\end{proof}


\section{Polish groups with metrizable universal minimal flows}
\label{sec:metrizable-flows}

Now we turn to proving Theorem~\ref{t:met1}. The proof is based on two propositions, \ref{p:precpct} and \ref{p:ext-am} below. It is interesting to note that even though the arguments we present here use only tools from topological dynamics, the original proof of Proposition~\ref{p:ext-am} was combinatorial and based on Ramsey theory.

Recall that if $X$ is a uniform space, its \df{Samuel compactification} $S(X)$ is the Gelfand space of the C$^*$-algebra $\UCB(X)$ of uniformly continuous bounded functions on $X$; it can equivalently be defined by the following universal property: if $f \colon X \to K$ is a uniformly continuous map to a compact Hausdorff space $K$, then $f$ extends uniquely to a map $S(X) \to K$. In the special case where $X = G/H$ with the uniformity defined by \eqref{eq:unif}, $S(X)$ is also a $G$-flow (one easily checks that the action of $G$ on the C$^*$-algebra $\UCB(G/H)$ is continuous) and we have the following lemma.

\begin{lemma} \label{l:S-univ}
Let $G$ be a topological group and $H$ a closed subgroup. Then the Samuel compactification $S(G/H)$ of the uniform space $G/H$ has the following universal property: for every $G$-ambit $(X, x_0)$ such that $H \cdot x_0 = x_0$, there exists a unique $G$-map $\psi \colon S(G/H) \to X$ such that $\psi(gH) = g \cdot x_0$ for every $g \in G$.
\end{lemma}
\begin{proof}
To prove the lemma, it suffices to observe that the map
\[
\Psi \colon C(X) \to \UCB(G/H), \quad \Psi(f)(gH) = f(g \cdot x_0)
\]
is a $G$-embedding of C$^*$-algebras.
\end{proof}

Of course, $G/G^*$ is precompact iff $S(G/G^*) = \widehat{G/G^*}$ (this is because if $G/G^*$ is precompact, then $C(S(G/G^*)) = \UCB(G/G^*) = C(\widehat{G/G^*})$).
\begin{prop} \label{p:precpct}
Let $G$ be a Polish group and $M(G)$ its universal minimal flow. Suppose that $M(G)$ is metrizable and that there is a point $x_0 \in M(G)$ such that the orbit $G \cdot x_0$ is $G_\delta$. If $G^*$ is the stabilizer of $x_0$, then $G/G^{*}$ is precompact and the $G$-spaces $\widehat{G/G^{*}}$ and $M(G)$ are isomorphic. 
\end{prop}
\begin{proof}
Let $Y$ denote the uniform space $G/G^*$ and $y_0 = G^*$. Let $S(Y)$ be the Samuel compactification of $Y$ and $i \colon Y \to S(Y)$ denote the natural embedding. Let $j \colon Y \to M(G)$ be the map given by $gG^* \mapsto g \cdot x_0$. By Effros's theorem (see \cite{Hjorth2000}*{Theorem~7.12}), $j$ is a homeomorphism onto its image. By Lemma~\ref{l:S-univ}, there exists a continuous $G$-map $\psi \colon S(Y) \to M(G)$ such that $\psi \circ i = j$. By the universal property of $M(G)$, there exists a continuous $G$-map $\phi \colon M(G) \to S(Y)$, so that we obtain the following diagram
\[ \xymatrix{
M(G) \ar@/_/[drr]_\phi \\
Y \ar[u]^j \ar[rr]_i & & S(Y) \ar@/_/[ull]_\psi.
} \]

We are going to show that $\phi$ is surjective. Let $Z = \phi(M(G))$ and $z_0 = \phi(x_0)$. Let $f = \phi \circ \psi$. As $\psi \circ \phi$ is an endomorphism of $M(G)$ and $M(G)$ is coalescent, $\phi \colon M(G) \to Z$ and hence $f|_Z \colon Z \to Z$ must be isomorphisms. Therefore by precomposing $\phi$ with an automorphism of $M(G)$, we can assume that $f|_Z = \id$. As the orbit $G \cdot i(y_0)$ is dense in $S(Y)$, there exists a net $(g_\alpha)_\alpha $ of elements of $G$ such that $g_\alpha \cdot i(y_0) \to z_0$. Applying $\psi$ to both sides, we get $g_\alpha \cdot \psi(i(y_0)) \to \psi(z_0)$. Because $\psi \circ i = j$, we have $\psi(i(y_0))=x_{0}$. Next, because $\psi\circ \phi$ is the identity on $M(G)$ and $z_{0}=\phi(x_{0})$, we get $\psi(z_{0})=x_{0}$. Thus, $g_{\alpha}\cdot x_{0}\to x_{0}$. Applying $i\circ j^{-1}$ to both sides of the previous limit, we get $g_\alpha \cdot i(y_0) \to i(y_{0})$, whence $z_0 = i(y_0)$ and $\phi$ is surjective. As we already saw that $\phi$ is an isomorphism onto its image, $M(G)$ is isomorphic to $S(Y)$. In particular, $S(Y)$ is metrizable, showing that $G/G^*$ is precompact and $M(G) \cong S(G/G^*) = \widehat{G/G^*}$.
\end{proof}

\begin{prop} \label{p:ext-am}
Let $G^* \le G$ be Polish groups such that $G/G^*$ is precompact and $\widehat{G/G^*}$ is the universal minimal flow of $G$. Then $G^*$ is extremely amenable. 
\end{prop}
\begin{proof}
  Fix a right-invariant metric $d_R$ on $G$. First note that the space $\mcF$ of $1$-Lipschitz functions $(G, d_R) \to [0, 1]$ endowed with the pointwise convergence topology is compact metrizable and equipped with the $G$-action
\[
(g \cdot \gamma)(x) = \gamma(xg),
\]
it becomes a $G$-flow.
  
Our goal is to show that the right translation action $G^* \actson (G^*,d_R)$ is \emph{finitely oscillation stable}, which is equivalent to saying that, for any $1$-Lipschitz $\gamma \colon (G^*,d_R) \to [0,1]$, there exists a $G^*$-fixed point in $\cl{G^* \cdot \gamma}$. By \cite{Pestov2006}*{Theorem 2.1.11}, this will imply that $G^*$ is extremely amenable.
We begin with a $1$-Lipschitz $\gamma \colon (G^*,d_R) \to [0,1]$, which we extend to a $1$-Lipschitz map from $(G,d_R)$ to $[0,1]$, still denoted by $\gamma$; for instance, one can achieve this by setting $\gamma(g)= \min(1,\inf \{\gamma(g^*)+d_R(g,g^*) \colon g^* \in G^*\})$.
Consider the diagonal action of $G$ on $\mathcal F \times \widehat{G/G^*}$. To avoid confusion, let $x_0$ denote the point $G^*$ in $\widehat{G/G^*}$. 

Since $M(G)$ has a $G^*$-fixed point, this is true for every $G$-flow; in particular, there is an $G^*$-fixed point $(\gamma_0,y)$ in $\overline{G \cdot (\gamma,x_0)}$. Then $y$ is an $G^*$-fixed point in $\widehat{G/G^*}$, so by Lemma~\ref{l:fixed-pts}, we know that $y=aG^*$ for some $a$ in the normalizer of $G^*$. Then the point $a^{-1} \cdot (\gamma_0,y) = (\gamma_1,x_0)$ is also fixed by $G^*$ and belongs to $\overline {G \cdot (\gamma, x_0)}$. Thus there exists a sequence $(g_i)_{i}$ of elements of $G$ such that $g_i \cdot \gamma \to \gamma_1$ and $g_i G^* \to G^*$. 

We can find a sequence $(h_i)_{i}$ of elements of $G^*$ such that $h_i g_i^{-1} \to 1_G$.
For any fixed $f \in G$, we have that
\[
|\gamma(f h_i)- \gamma(f g_i)| \leq d_R(f h_i, f g_i) = d_R(f h_i g_i^{-1} f^{-1}, 1_G) \to 0,
\]
so that $\gamma(f h_i) \to \gamma_1(f)$. Thus, the $G^*$-fixed point $\gamma_1$ belongs to $\overline{G^* \cdot \gamma}$, concluding the proof.
\end{proof}

We are now ready to prove Theorem~\ref{t:met1} from the introduction.
\begin{proof}[Proof of Theorem~\ref{t:met1}]
\ref{met1:i:1}$\Rightarrow$\ref{met1:i:2}. Let $x_{0}\in M(G)$ be a point with a generic orbit and let $G^{*}$ denote the stabilizer of $x_{0}$. By Proposition \ref{p:precpct}, $M(G) = \widehat{G/G^{*}}$. By Proposition \ref{p:ext-am}, $G^{*}$ is extremely amenable. 

\ref{met1:i:2}$\Rightarrow$\ref{met1:i:1}. Immediate after observing that $G/G^{*}$ is a generic orbit in $\widehat{G/G^{*}}$.
\end{proof}

\subsection{Smoothness of isomorphism}
\label{sec:smooth-isom}

An important direction in modern descriptive set theory is classifying definable equivalence relations according to their complexity. This is done mostly via the notion of Borel reducibility: an equivalence relation $E$ on a Polish space $X$ is \df{Borel reducible} to an equivalence relation $F$ on a Polish space $Y$ if there exists a Borel map $f \colon X \to Y$ such that for all $x_1, x_2 \in X$, $x_1 \eqrel{E} x_2 \iff f(x_1) \eqrel{F} f(x_2)$. Many natural examples arise as equivalence relations of isomorphism of various mathematical objects. One parametrizes the objects of interest by the elements of some Polish space and then tries to understand how complex the equivalence relation of isomorphism is. An equivalence relation is called \df{smooth} if it is Borel reducible to equality on some Polish space. Smooth equivalence relations are the simplest ones and they are at the bottom of the Borel reducibility hierarchy. For more on the theory of Borel reducibility, see, for example, \cite{Gao2009a}.

If the universal minimal flow of a Polish group $G$ is metrizable and has a $G_\delta$-orbit, this has quite strong implications about \emph{all} minimal flows of $G$. In this subsection, we show that isomorphism of minimal flows of such a group is smooth. This should be contrasted with a recent non-classification result, due to Gao, Jackson and Seward~\cite{Gao2012p}, stating that isomorphism of minimal subshifts of an infinite, countable group is \emph{not} smooth.

Let $G$ be a Polish group such that $M(G)$ is metrizable. If $X$ is any minimal flow of $G$, then there is a $G$-map $\pi \colon M(G) \to X$ that gives rise to an equivalence relation $\mcR_\pi$ on $M(G)$ defined by
\[
z_1 \eqrel{\mcR_\pi} z_2 \iff \pi(z_1) = \pi(z_2).
\]
$\mcR_\pi$ is a $G$-invariant, closed, equivalence relation, \df{icer} for short. Conversely, every icer $\mcR$ defines a minimal flow $M(G)/\mcR$. We have the following general fact.
\begin{prop}
  \label{p:equiv-Polish}
  Let $Y$ be a compact Polish space and let
  \[
  \mcE = \set{ \mcR \sub Y^2 : \mcR \text{ is a closed equivalence relation}}.
  \]
  Then $\mcE$ is $G_\delta$ in $K(Y^2)$, the space of compact subsets of $Y^2$ equipped with the Vietoris topology, and therefore a Polish space.
\end{prop}
\begin{proof}
  The conditions that $\mcR$ is reflexive and symmetric are closed. We check that transitivity is $G_\delta$. We have that
  \[
  \mcR \text{ is not transitive} \iff \exists x, y, z \in Y \ (x, y) \in \mcR \And (y, z) \in \mcR \And (x, z) \notin \mcR.
  \]
  The set of $(x, y, z, \mcR) \in Y^3 \times K(Y^2)$ that satisfy the condition after the quantifier is an intersection of an open and a closed set in a compact space, so it is $K_\sigma$ (a countable union of compact sets). Therefore its projection on $K(Y^2)$ is also $K_\sigma$, hence it is $F_\sigma$, and transitivity is $G_\delta$.
\end{proof}

Let
\[
\mcE = \set{\mcR \sub M(G)^2 : \mcR \text{ is an icer}}.
\]
In view of the preceding discussion and Proposition~\ref{p:equiv-Polish}, it is natural to parametrize minimal flows of $G$ by elements of the Polish space $\mcE$ (being $G$-invariant is obviously a closed condition).
\begin{theorem}
  \label{th:smooth-isom}
  Let $G$ be a Polish group and suppose that $M(G)$ is metrizable and has a $G_\delta$ orbit. Let $\mcR_1$ and $\mcR_2$ be icers on $M(G)$. Then $M(G)/\mcR_1$ is isomorphic to $M(G)/\mcR_2$ iff there is $\sigma \in \Aut(M(G))$ such that $\sigma \cdot \mcR_1 = \mcR_2$. In particular, isomorphism of minimal flows of $G$ is smooth.
\end{theorem}
\begin{proof}
  The $(\Leftarrow)$ direction being obvious, we set to prove the converse. By Theorem~\ref{t:met1}, $M(G) = \widehat{G/G^*}$ for some co-precompact, extremely amenable $G^* \leq G$. Let $z_0$ be the distinguished $G^*$-fixed point of $\widehat{G/G^*}$. Let $\pi_i \colon M(G) \to M(G)/\mcR_i$ denote the natural factor maps and finally, let $f \colon M(G)/\mcR_1 \to M(G)/\mcR_2$ be an isomorphism. The point $x_2 = f(\pi_1(z_0)) \in M(G)/\mcR_2$ is a fixed point for $G^*$. Let $K = \pi_2^{-1}(\set{x_2})$. Then $K$ is a $G^*$-invariant closed subset of $M(G)$ and by the extreme amenability of $G^*$, there is a $G^*$-fixed point $z_1 \in K$. By Lemma~\ref{l:fixed-pts} and Theorem~\ref{th:coalescent}, there exists $\sigma \in \Aut(M(G)) = N(G^*)/G^*$ such that $\sigma(z_0) = z_1$. Then $f(\pi_1(z_0)) = \pi_2(\sigma(z_0))$, so $f \circ \pi_1 = \pi_2 \circ \sigma$ everywhere. This shows that $\sigma \cdot \mcR_1 = \mcR_2$.

  For the second conclusion, recall that $\Aut(M(G))$ is a compact group acting continuously on the Polish space $\mcE$, so $\mcE/\Aut(M(G))$ is a Polish space.
\end{proof}

\section{Universal minimal proximal flows and strong amenability}
\label{sec:proximal-flows}

\subsection{Universal minimal proximal flows}
Two points $x, y$ in a $G$-flow $X$ are called \df{proximal} if there exists a net $(g_{\alpha})_\alpha$ of elements of $G$ such that $\lim_\alpha g_\alpha \cdot x = \lim_\alpha g_\alpha \cdot y$. The points $x$ and $y$ are \df{distal} if $x = y$ or $x$ and $y$ are not proximal. The flow $X$ is called \df{proximal} if every pair of points is proximal and \df{distal} if every pair of points is distal. Every topological group $G$ admits a \df{universal minimal proximal flow} $\Pi(G)$ (i.e., one that maps onto every other minimal proximal flow) that is unique up to isomorphism. For a proof of this fact and more background on distal and proximal flows, see \cite{Glasner1976}. 

If $\gamma \in \Aut(X)$, then the points $x$ and $\gamma(x)$ are distal for every $x \in X$ (see \cite{Glasner1976}*{II, Lemma~3.3}). As a consequence, in order to construct a proximal factor of $M(G)$, it is necessary to divide by the action of $\Gamma = \Aut(M(G))$. It turns out that this is also sufficient to yield the universal minimal proximal flow.

When $X$ is a $G$-flow and $\Gamma \leq \Aut(X)$, we define its \df{maximal $\Gamma$-invariant factor $Z$} as follows: let $R$ denote the smallest closed equivalence relation containing all $\{(x,\gamma(x)) \colon \gamma \in \Gamma \}$. Then $Z$ is the quotient of $X$ by $R$, endowed with the quotient topology, which is compact Hausdorff because $R$ is closed. Observe that $G$ naturally acts on $Z$, and the factor map is $G$-equivariant.

In the special case where $\Gamma=\Aut(M(G))$ is compact, $\set{(x, \gamma(x)) : x \in M(G), \gamma \in \Gamma}$ is already closed and the quotient by it, denoted by $M(G)/\Gamma$, is the maximal $\Gamma$-invariant factor of $M(G)$. By Theorem~\ref{th:coalescent}, this always happens when $M(G)$ is of the form $\widehat{G/G^*}$. The following theorem implies in particular that, in that case, the universal proximal flow of $G$ is $M(G)/\Gamma$.

\begin{theorem} \label{th:univ-prox-flow}
Let $G$ be a topological group, $M(G)$ be its universal minimal flow, and $\Gamma = \Aut(M(G))$. Then the universal proximal flow of $G$ is the maximal $\Gamma$-invariant factor of $M(G)$.
\end{theorem}
\begin{proof}
  Let $Z$ be the maximal $\Gamma$-invariant factor of $M(G)$ and let $\sigma \colon M(G) \to Z$ be the factor map. By definition, $\sigma(\gamma (x)) = \sigma(x)$ for all $\gamma \in \Gamma, x \in M(G)$. To verify that $Z$ is proximal, it suffices to see that for every point $(x, y) \in M(G)^2$, there is $\gamma \in \Gamma$ such that $x$ and $\gamma(y)$ are proximal. To that end, consider an \df{almost periodic} point (i.e., one with a minimal orbit closure) $(x', y') \in \overline{G \cdot (x, y)}$ and a net $(g_\alpha)_{\alpha} \sub G$ such that $g_\alpha \cdot (x, y) \to (x', y')$. Then, we claim that there exists $\gamma \in \Gamma$ such that $y' = \gamma (x')$. Indeed, let $X = \overline{G \cdot (x', y')}$ and $\pi_1, \pi_2 \colon X \to M(G)$ be the two canonical projections. As $M(G)$ is universal and coalescent, $\pi_1$ is an isomorphism. It follows that $\pi_{1} ^{-1}(x') = (x',y')$. Similarly, $\pi_{2}$ is an isomorphism, and so $\gamma = \pi_2\pi_1^{-1}$ is an automorphism of $M(G)$ which satisfies $\gamma (x') = y'$. Now,
\[
g_\alpha \cdot (x, \gamma^{-1} (y)) = (g_\alpha \cdot x, \gamma^{-1}(g_\alpha \cdot y)) \to (x', \gamma^{-1}(y')) = (x', x'),
\]
whence $x$ and $\gamma^{-1}(y)$ are proximal.

We next check that $Z$ is universal for proximal flows. Let $X$ be a proximal flow. As $M(G)$ is universal, there is a map $\pi \colon M(G) \to X$. We need to check that $\pi$ factors through $Z$, i.e., $\pi(y) = \pi(\gamma(y))$ for every $y \in M(G)$, $\gamma \in \Gamma$. The points $\pi(y)$ and $\pi(\gamma(y))$ are proximal in $X$, so there is $x \in X$ and a net $(g_\alpha)_{\alpha} \sub G$ such that $g_\alpha \cdot (\pi(y), \pi(\gamma(y))) \to (x, x)$. By passing to a subnet, we can assume that $g_\alpha \cdot y$ and $g_\alpha \cdot \gamma(y) = \gamma(g_{\alpha} \cdot y)$ converge to $y'$ and $\gamma(y')$, respectively. By applying $\pi$ to both sides, we see that $\pi(y') = x = \pi(\gamma (y'))$. As $M(G)$ is minimal, there is a net $(h_\alpha)_{\alpha}$ such that $h_\alpha \cdot y' \to y$. Then $h_\alpha \cdot \pi(y') \to \pi(y)$ and $h_{\alpha} \cdot \pi(\gamma (y')) \to \pi(\gamma(y))$. So $\pi(y) = \pi(\gamma (y))$ as required. 
\end{proof}

Recall that a topological group $G$ is called \df{strongly amenable} if it admits no non-trivial minimal proximal flows, or equivalently, $\Pi(G)$ is a singleton. This notion is a strengthening of the notion of amenability, see \cite{Glasner1976}*{II.3} for more details.

\begin{prop}
\label{p:UMPF}
Let $G$ be a Polish group and assume that $M(G) = \widehat{G/G^{*}}$ with $G^{*}$ a closed, co-precompact subgroup of $G$. Then $N(G^{*})$ is strongly amenable and $\Pi(G) = \widehat{G/N(G^{*})}$.
\end{prop}
\begin{proof}
  As $G^{*}$ is extremely amenable (by Theorem~\ref{t:met1}), the $N(G^{*})$-flow $N(G^{*})/G^{*}$ is the universal minimal flow of $N(G^{*})$. This flow, being the translation on a compact group, is distal. As every proximal factor of a distal flow is trivial (see, for example, \cite{Glasner1976}*{II, Corollary 1.3}), this implies that $N(G^{*})$ is strongly amenable.

  For the second statement, note that, by Theorem~\ref{th:coalescent} and Theorem~\ref{th:univ-prox-flow},
  \[
  \Pi(G) = M(G) / \Aut(M(G)) = \widehat{G/G^*} \big / (N(G^*)/G^*) = \widehat{G/N(G^*)}.
  \]
To see why the last equality holds, observe that the natural map $G/G^* \big / (N(G^*)/G^*) \to G/N(G^*)$ is a uniform isomorphism and therefore extends to the completions. As the group $N(G^*)/G^*$ is compact and acts by isometries, the completion of $G/G^* \big / (N(G^*)/G^*)$ is isomorphic to $\widehat{G/G^*} \big / (N(G^*)/G^*)$.
\end{proof}

From Proposition~\ref{p:UMPF}, we directly deduce the following corollary.
\begin{cor}
\label{c:SA1}
Let $G$ be a Polish group and assume that $M(G) = \widehat{G/G^{*}}$ with $G^{*}$ a closed, co-precompact subgroup of $G$. Then $G$ is strongly amenable iff $G^{*}$ is normal in $G$ iff $M(G)$ is a compact group quotient of $G$.
\end{cor}


As a consequence of the above Corollary, any strongly amenable Polish group such that $M(G)$ is metrizable with a generic orbit is uniquely ergodic in the sense of \cite{Angel2012}, as in that case, the unique invariant probability measure on $M(G)$ is the Haar measure. 

\subsection{A characterization of strong amenability in terms of invariant measures}
As indicated previously, strong amenability for a topological group $G$ is a strengthening of the notion of amenability. However, no equivalent reformulation of it in terms of the existence of probability measures on $G$-flows with certain properties seems to be known. In this subsection, we provide such a characterization under the assumption that $M(G) = \widehat{G/G^{*}}$ with $G^{*}$ closed and co-precompact.  

Let $\mcM(X)$ (resp. $\mcM _{\mathrm{fin}}(X)$) denote the set of all Borel (resp. finitely supported) probability measures on $X$. Recall that $\mcM(X)$ is compact if equipped with the weak$^*$ topology, and that its unique compatible uniformity is generated by the family $\mcU_X$ consisting of all finite intersections of closed sets of the form
\[
\big\{ (\mu,\nu) \in \mcM(X)^{2} : \Big|\int_{X}f \ud\mu - \int_{X}f \ud\nu \Big|\leq \varepsilon \big\},
\]
where $f$ is a continuous function on $X$ and $\varepsilon > 0$.

\begin{prop}\label{p:strongmes}
Let $G$ be a Polish group and assume that $M(G) = \widehat{G/G^{*}}$ with $G^{*}$ a closed, co-precompact subgroup of $G$. Then the following are equivalent: 
\begin{enumerate}
\item \label{i:4} $G$ is strongly amenable;
\item \label{i:5} For every $G$-flow $X$ and $\eps \in \mcU_X$, there exists $\nu \in \mcM _{\mathrm{fin}}(X)$ such that 
\begin{equation*} 
\forall g_1, g_2 \in G \quad (g_1 \cdot \nu, g_2 \cdot \nu) \in \eps.
\end{equation*}
\end{enumerate}
\end{prop}
\begin{proof}
\ref{i:5}$\Rightarrow$\ref{i:4}. Let $G\actson X$ be proximal. For $\eps \in \mcU_X$, let 
\[
A_\eps = \set{\mu \in \mcM(X) : \forall g_1, g_2 \in G \ (g_1 \cdot \mu, g_2 \cdot \mu) \in \eps}.
\] 
This set is $G$-invariant and closed in $\mcM(X)$, and by hypothesis, contains a finitely supported measure $\nu$. By proximality, we can find a net $(g_{\alpha})_{\alpha}$ of elements in $G$ and $x \in X$ such that for every $y$ in the support of $\nu$, $g_{\alpha}\cdot y \to x$. Thus, $g_{\alpha}\cdot \nu \to \delta_{x}$, the Dirac measure at $x$, and $A_\eps \cap X$ is not empty (here, $X$ is identified with the set of Dirac measures in $\mcM(X)$). By compactness, $\bigcap_\eps A_\eps \cap X\neq \emptyset$ and it remains to observe that any element of this set is fixed by $G$.

\ref{i:4}$\Rightarrow$\ref{i:5}. Let $G\actson X$ be a $G$-flow, which, without loss of generality, we may assume to be minimal. Let $\eps \in \mcU_X$ and let $K$ denote the compact group $G/G^{*}$. By Corollary~\ref{c:SA1}, the action of $G$ factors through $K$ and it suffices to find a finitely supported Borel probability measure $\nu$ such that for all $k_1, k_2 \in K$, $(k_1 \cdot \nu, k_2 \cdot \nu) \in \eps$. This is easy: any finitely supported measure close enough to the push-forward of the Haar measure of $K$ will do.
\end{proof}

Note that the assumption $M(G) = \widehat{G/G^{*}}$ is not used in the proof of \ref{i:5}$\Rightarrow$\ref{i:4}.

\subsection{Strong amenability for automorphism groups}
Another motivation for the material of the previous section was to extract a Ramsey-type statement from strong amenability when $G$ is a subgroup of $S_{\infty}$. This is the purpose of what follows. 

\begin{lemma}
Let $G = \Aut(\bF)$ for an $\omega$-categorical, Fraïssé structure $\bF$ and assume that $M(G) = \widehat{G/G^{*}}$ with $G^{*}$ a closed, co-precompact subgroup of $G$. Assume also that $G^{*} = \Aut(\bF^{*})$ with $\bF^{*}=(\bF, R_1, \ldots, R_k)$ a Fraïssé expansion of $\bF$ by finitely many relation symbols. Then $N(G^{*})/G^{*}$ is finite.
\end{lemma}
\begin{proof}
  Note that the coset $gG^*$, $g \in G$ is determined by the relations $g \cdot R_1, \ldots g \cdot R_k$. If $g \in N(G^*)$, then $g \cdot R_i$ is a $G^*$-invariant relation. Since $G^*$ is co-precompact and the action of $G$ on $\bF$ is oligomorphic by assumption, this is also the case for the action of $G^*$, i.e. there are only finitely many orbits for the natural action of $G^*$ on $F^k$, for all $k$. This implies that there are only finitely many $G^*$-invariant relations in every arity, whence we conclude that $N(G^*)/G^*$ must be finite.
\end{proof}

We now have the following corollary of Proposition~\ref{p:UMPF}.
\begin{cor}
Under the assumptions of the previous lemma, $G$ is strongly amenable iff $M(G)$ is finite.  
\end{cor}

\begin{proof}
Since any finite flow is distal, a topological group with a finite universal minimal flow must be strongly amenable. The other direction is obvious from the previous lemma and Proposition~\ref{p:UMPF}. 
\end{proof}

One reason for us to study strong amenability was to determine its Ramsey-theoretic content. When the hypotheses above are satisfied, the previous corollary makes this possible via the work \cite{Mueller2014p} of M\"uller and Pongr\'acz, where it is shown that $M(G)$ is finite with size at most $d$ iff in $\Age(\bF)$, Ramsey degrees for embeddings are all finite and at most $d$. In the slightly more general case where $\bF^{*}$ is a Fraïssé expansion with infinitely many symbols, this is still doable thanks to Proposition~\ref{p:strongmes}, but we do not detail this here.
However, in contrast with amenability (see \cite{Moore2013} for details), it is unclear what can be said without an assumption on $M(G)$.

\section{Examples}
\label{sec:examples}
In this section, we calculate the universal minimal proximal flows for some concrete examples. We will consider subgroups $G \leq S_\infty$ that are given as automorphism groups of some Fraïssé limit $\bF$, and will assume some familiarity with the framework of \cite{Kechris2005}. We will cover the cases where $\bF$ is a homogeneous graph or tournament, but thanks to the results of \cite{Jasinski2014}, a similar strategy can be used for all homogeneous directed graphs. In view of Theorem~\ref{t:met1}, in order to have $M(G) = \widehat{G/G^*}$, it is necessary and sufficient that there exist a closed, co-precompact, extremely amenable subgroup $G^* \leq G$ such that the $G$-flow $\widehat{G/G^*}$ is minimal. In the setting of Fraïssé limits, this is equivalent to the existence of a precompact, homogeneous expansion $\bF^*$ of $\bF$ that has the Ramsey property and the expansion property \cites{Kechris2005,NguyenVanThe2013}.

\subsection{Betweenness relations}
We first consider the (rather common) situation where the Ramsey expansion is obtained by adding a linear order $<$ to the signature and taking the limit of an appropriate Fraïssé class. Then one can naturally define the corresponding betweenness relation $B$ by
\begin{equation} \label{eq:bet}
B(x,y,z) \iff (x<y<z \, \vee \, z<y<x).
\end{equation}
This relation is relevant for us because in a number of cases, the normalizer of $\Aut(\bF, <)$ in $\Aut(\bF)$ is $\Aut(\bF, B)$. More precisely, we have the following.
\begin{lemma}\label{l:B}
  Assume that $\bF$ is a homogeneous structure and that there is an order Fraïssé expansion $\bF^* = (\bF, <)$ so that $M(G) = \widehat{G/G^{*}}$, where $G=\Aut(\bF)$ and $G^{*}=\Aut(\bF^*)$. Assume also that there exist $u, v\in \bF^*$ such that $u < v$ and
\begin{equation} \label{eq:3}
\begin{split}
  \forall x < y \in \bF^* \ \exists x_{0},...,x_{n+1} \in \bF^* \quad \big( &x_{0}=x \And x_{n+1}=y \And \\
  & \exists g_{0}, \ldots, g_n \in G^{*} \forall i \ g_{i}(u)=x_{i} \And g_{i}(v)=x_{i+1} \big).
\end{split}
\end{equation}
Then $N(G^{*})=\Aut(\bF, B)$, where $B$ is the betweenness relation defined by \eqref{eq:bet}. 
\end{lemma}
\begin{proof}
  Observe that an element $g \in G$ is in $\Aut(\bF, B)$ iff it preserves $<$ or it \df{reverses} $<$, i.e., for all $x < y \in \bF$, $g \cdot x > g \cdot y$, so in particular, $G^*$ has index at most $2$ in $\Aut(\bF, B)$ and is normal there.
  
  Conversely, suppose that $g\in N(\Aut(\bF^*))$ and let $u$ and $v$ be as in the hypothesis. Let $x < y \in \bF$ be arbitrary and let $x_0, \ldots, x_{n+1}$ be as in \eqref{eq:3}. Then for every $i \leq n$,
  \[
  g \cdot u < g \cdot v \iff (g g_i g^{-1}) g \cdot u < (g g_i g^{-1})g \cdot v
    \iff g \cdot x_i < g \cdot x_{i+1},
  \]
  and by transitivity of the ordering, we obtain that
  \[
  \forall x, y\in \bF \quad g \cdot u < g \cdot v \iff g \cdot x < g \cdot y,
  \]
i.e., the ordering $g \cdot <$ is either equal to $<$ or to its inverse, depending on whether $g \cdot u < g \cdot v$ or not. In particular, this means that $g \in \Aut(\bF, B)$.
\end{proof}

Now we consider the case where the universal minimal flow of $G$ is the space of \emph{all} linear orderings on $\bF$. Denote by $\mathrm{BLO}(\bF)$ the set of all betweenness relations on $\F$ that are induced by some linear ordering in $\LO(\bF)$ by the formula \eqref{eq:bet}. Then it is not difficult to check that $\mathrm{BLO}(\bF)$ is a closed subset of $2^{\bF^3}$ and the completion of $G/\Aut(\bF, B)$ is isomorphic to $\mathrm{BLO}(\bF)$ (equipped with the logic action of $G$). Therefore Proposition~\ref{p:UMPF} yields the following.
\begin{cor}
  Suppose that $\bF$ and $\bF^* = (\bF, <)$ are Fraïssé structures that satisfy the hypothesis of Lemma~\ref{l:B} and suppose moreover that $M(G) = \LO(\bF)$. Then $\Pi(G) = \mathrm{BLO}(\bF)$.
\end{cor}

The last corollary applies to many structures, including the structure in the empty language, the random graph, all Henson graphs, the rational Urysohn space, and the random tournament (see \cite{Kechris2005} for the calculation of the universal minimal flows).

One case where it does \emph{not} apply is the countable generic poset $\bP$. Denote by $\mathrm{LE}(\bP)$ the space of all linear orderings that extend the poset relation $\leq^{\bP}$. This is of course a proper subset of $\mathrm{LO}(\bP)$. It was shown in \cite{Kechris2005} (see the Addendum at the end) that $\mathrm{LE}(\bP)$ is the universal minimal flow of $G=\Aut(\bP)$. We are in the situation of Lemma \ref{l:B}, where $\bF^*$ is of the form $(\bP,\le^\bP,\preceq)$, the Fra\"iss\'e limit of all finite posets endowed with an ordering extending the poset relation.
Denoting as usual by $G^*$ the automorphism group of $(\bP,\le^\bP,\preceq)$, the reasoning of Lemma \ref{l:B} (the converse direction) shows that any $g$ belonging to the normalizer of $G^*$ in $G$ must either preserve or reverse $\preceq$; clearly, in this case, there are no elements of $G$ which reverse $\preceq$, so we obtain that $G^*$ is normal in $G$.
Proposition~\ref{p:UMPF} now tells us that $\Pi(\Aut(\bP)) = M(\Aut(\bP))$, i.e., every minimal flow of $\Aut(\bP)$ is proximal.

Another group for which all minimal flows are proximal is the automorphism group of the countable atomless Boolean algebra (see \cite{Glasner2003a}); we are grateful to one of the referees for pointing this out.

\subsection{Ultrahomogeneous graphs}
We already computed $\Pi(G)$ in the case of the automorphism groups of the infinite complete graph $K_{\N}$, the Henson graphs and the random graph. The remaining cases of countable ultrahomogeneous graphs are, up to a switch of the edges and the non-edges: 

\begin{enumerate}
\item \label{i:1} $I_{n}[K_{\N}]$, made of $n$ many disjoint copies of $K_{\N}$, where $n\in \N$ is fixed;
\item \label{i:2} $I_{\N}[K_{n}]$, made of infinitely many disjoint copies of $K_{n}$, where $n\in \N$ is fixed;
\item \label{i:3} $I_{\N}[K_{\N}]$, made of infinitely many disjoint copies of $K_{\N}$.
\end{enumerate}

\ref{i:1}. $G=S_{n}\ltimes S_{\infty}^{n}$, $S_{n}$ being the symmetric group of $[n]$ acting on $S_{\infty}^{n}$ by permuting the coordinates. In that case, $G^{*}=\Aut(I_{n}[K_{\N}]^{*})$, where $I_{n}[K_{\N}]^{*}$ is obtained by adding $n$ many unary predicates $A^{*}_{i}$, $i\in [n]$ (one for each copy of $K_{\N}$), and by adding a linear ordering $<$ so that each $A^{*} _{i}$ is isomorphic to $(\Q, <)$ and $A^{*}_{0} < A^{*}_{1}<...<A^{*} _{n-1}$. So $G^{*}=\Aut(\Q,<)^{n}$ and $N(G^{*})=S_{n}\ltimes \Aut(\Q,B)^{n}$. Therefore, $\Pi(G)$ is given by the natural action of $S_{n}\ltimes S_{\infty}^{n}$ on $\mathrm{BLO}(\N)^{n}$.  

\ref{i:2}. $G=S_{\infty}\ltimes S_{n} ^{\N}$, where $S_{\infty}$ acts on $S_{n} ^{\N}$ by permuting the coordinates. For convenience, we will see this group as $S_{\infty}\ltimes S_{n}^\Q $. Then, $G^{*}=\Aut(I_{\N}[K_{n}]^{*})$, where $I_{\N}[K_{n}]^{*}$ is obtained by adding a linear ordering that leaves the copies of $K_{n}$ convex and orders the set of copies as $(\Q, <)$. 
The corresponding group $G^{*}$ is $\Aut(\Q, <)$ and corresponds to the set of permutations that only move the copies of $K_{n}$ in an order-preserving way and no non-trivial permutations are allowed within the copies. From this, we see that an element of $N(G^{*})$ may reverse the ordering on the set of copies, and inside each copy may permute the elements according to a common $\sigma \in S_{n}$. Hence, denoting by $\Delta(S_{n}^\Q)$ the diagonal image of $S_{n}$ in $S_{n}^\Q$, we obtain $N(G^{*})=\Aut(\Q, B)\ltimes \Delta(S_{n}^\Q)\subseteq S_{\infty}\ltimes S_{n}^\Q$. Therefore, the flow $\Pi(G)$ can be identified with the natural action of $S_{\infty}\ltimes S_{n}^\Q$ on $\mathrm{BLO}(\N)\times (S_{n}^\Q  /\Delta(S_{n}^\Q))$.

\ref{i:3}. $G=S_{\infty}\ltimes S_{\infty} ^{\N}$, which, as previously, we will see as $S_{\infty}\ltimes S_{\infty}^\Q$ where $S_{\infty}$ acts on $S_{n}^\Q$ by permuting the coordinates. Then $G^{*}=\Aut(I_{\N}[K_{\N}]^{*})$, where $I_{\N}[K_{\N}]^{*}$ is obtained by adding a convex linear ordering $<^*$, so that $G^{*}=\Aut(\Q, <) \ltimes \Aut(\Q,<)^\Q$. An element $g$ of $N(G^{*})$ may reverse the ordering on the set of parts. Moreover, if $x<^{*}y$ are in the same copy of $K_{\N}$, then $g(x)<^{*}g(y)$ iff $g$ is order-preserving on each copy of $K_{\N}$, and $g(y)<^{*}g(x)$ iff $g$ reverses $<^{*}$ on each copy of $K_{\N}$.
Therefore, denoting by ($\mathrm{LO}(\N)^{\N}/\sim$) the set obtained from $\mathrm{LO}(\N)^{\N}$ by identifying each element $(<_{i})_{i}$ with its reverse version $(>_{i})_{i}$, the flow $\Pi(G)$ can be identified with the natural action of $S_{\infty}\ltimes S_{\infty}^{\N}$ on the set $\mathrm{BLO}(\N)\times(\mathrm{LO}(\N)^{\N}/\sim)$. 

\subsection{Ultrahomogeneous tournaments}
The three countable ultrahomogeneous tournaments are $(\Q,<)$, the random tournament, and the dense local order $\bS(2)$. The first two cases were considered in the previous sections, so the only remaining case to treat is $\bS(2)$. In what follows, we write $G$ for $\Aut(\bS(2))$. For this structure, it was shown \cite{NguyenVanThe2013} that $M(G) = \widehat{G/G^*}$, where $G^{*} = \Aut(\bS(2), P^{*} _{0}, P^{*} _{1}) \leq G$ and $P^{*} _{0}, P^{*} _{1}$ is the partition of $\bS(2)$ into right part and left part. Let $E^{*}$ denote the equivalence relation induced by the partition $(P^{*} _{0}, P^{*} _{1})$. 

\begin{lemma}
$N(G^{*})=\Aut(\bS(2), E^{*})$.
\end{lemma}
\begin{proof}
Let $x \eqrel{E^{*}} y \in \bS(2)$. Let $g\in N(G^{*})$ and $g^{*}\in G^{*}$ so that $g^{*}(x)=y$. Fix $j \in \set{0, 1}$ such that $g(x)\in P^{*} _{j}$. Then because $gg^{*}g^{-1}\in G^{*}$, we have $gg^{*}g^{-1}(g(x))\in P^{*} _{j}$, i.e., $g(y)\in P^{*} _{j}$. In other words, $g(x) \eqrel{E^{*}} g(y)$. So $N(G^{*})\subseteq\Aut(\bS(2), E^{*})$. The other inclusion is easy. 
\end{proof}

Proposition~\ref{p:UMPF} now yields the following.
\begin{cor}
$\Pi(\Aut(\bS(2))) = \overline{\Aut(\bS(2))\cdot E^{*}}$, where the closure is taken in $2^{\bS(2) \times \bS(2)}$.
\end{cor}

Note that as in the case of $M(\Aut(\bS(2)))$ in \cite{NguyenVanThe2013}, it is possible to make an explicit description of the space $\overline{\Aut(\bS(2))\cdot E^{*}}$: roughly, it is the space obtained from the unit circle, where points with rational angle are doubled, and where antipodal points are then identified.  

\bibliography{MetrizableUMF}
\end{document}